\documentclass[leqno]{amsart}

\usepackage{amsmath}
\usepackage{amsfonts}
\usepackage{amssymb}
\usepackage{pdfsync}
\usepackage{color,graphicx,tikz-cd}
\usepackage{a4wide}
\usepackage{amscd,amsthm,latexsym}
\usepackage{hyperref}
\usepackage{kotex}
\usepackage{relsize}
\usepackage{cite}

\allowdisplaybreaks %

\title{A proof of the Delta Conjecture when $q=0$}

\author{Adriano Garsia}
\address{Department of Mathematics, University of California, San Diego, La Jolla, CA, USA}
\email{garsia@math.ucsd.edu}

\author{Jim Haglund}
\address{
Department of Mathematics, University of Pennsylvania, Philadelphia, PA, USA}
\email{jhaglund@math.upenn.edu}

\author{Jeffrey B. Remmel}
\address{
Department of Mathematics, University of California, San Diego, La Jolla, CA, USA}
\email{jremmel@ucsd.edu}

\author{Meesue Yoo}
\address{
Applied Algebra and Optimization Research Center, Sungkyunkwan University, Suwon,
South Korea}
\email{meesue.yoo@skku.edu}

\date{\today}

\thanks{
The first  author was supported by NSF grant DMS-$1700233$.
The second author was supported by NSF grant DMS-$1600670$. The fourth author was supported by NRF grants $2016R1A5A1008055$ and $2017R1C1B2005653$.  
}

\keywords{Delta Conjecture, Macdonald polynomials}

\subjclass[2010]{Primary:05E05  ; Secondary: 05E10}

\date{\today}

\newtheorem{thm}{Theorem}[section]
\newtheorem{lem}[thm]{Lemma}
\newtheorem{prop}[thm]{Proposition}

\theoremstyle{definition}

\def \o1  {{\overline }}
\def \l {{\ell}}

\def\tttt #1{{\textstyle{#1} }}

\def \magstep#1 {\ifcase#1 1000\or 1200\or 1440\or 1728\or 2074\or 2488\fi\relax}

\def\la{{\lambda}}

\font\title=cmbx10 scaled\magstep2

\def \-> {\rightarrow}

\def\la {\lambda}

\def \RA {\rightarrow}

\def \ses {\enskip = \enskip}

\def \ess {\enskip}

\def \ssp {\hskip .25em}
\def \bigsp {\hskip .5truein}
\def \part {\vdash}

\def \RA {{ \rightarrow }}

\def \ess {\enskip}
\def \ssp {\hskip .25em}
\def \bigsp {\hskip .5truein}
\def \part {\vdash}

\begin{document}

%------------------------------------------------------------------------------------------------------

\begin{abstract}
In [The Delta Conjecture, {\it Trans. Amer. Math. Soc.}, to appear] Haglund, Remmel, Wilson introduce a conjecture which gives a combinatorial prediction for the result of
applying a certain operator to an elementary symmetric function.  This operator, defined in terms of its action on the modified Macdonald basis, has played a role in work of Garsia and Haiman on 
diagonal harmonics, the Hilbert scheme, and Macdonald polynomials
[A. M. Garsia and M. Haiman. A remarkable {$q,t$}-Catalan sequence and $q$-Lagrange inversion, {\it J. Algebraic Combin.} {\bf 5} (1996), 191--244],  
[M. Haiman, Vanishing theorems and character formulas for the Hilbert scheme of points in the plane,
{\it Invent. Math.} {\bf 149} (2002), 371-407]. The Delta Conjecture involves two parameters $q,t$; in this article we give the first proof that the Delta Conjecture is true when $q=0$ or $t=0$.  
\end{abstract}

%------------------------------------------------------------------------------------------------------

\maketitle
%\tableofcontents

%------------------------------------------------------------------------------------------------------

\section{Introduction}

%------------------------------------------------------------------------------------------------------

For any partition $\mu$, we let ${\widetilde H}_{\mu}(X;q,t)$ denote the modified Macdonald polynomial, and 
\begin{align}
B_{\mu}(q,t)  =\sum_{c\in \mu}q^{a'(c)}t^{l'(c)},
\end{align}
where the sum is over all squares $c$ in the Ferrers shape of $\mu$, $a^{\prime}(c)$, the coarm of $c$, is the distance to the left border of $\mu$, and
$l^{\prime}(c)$, the coleg of $c$, is the distance to the bottom border of $\mu$, as in 
Figure \ref{legarm}.

\begin{figure}[!ht]
\begin{center}
\includegraphics[scale=.5]{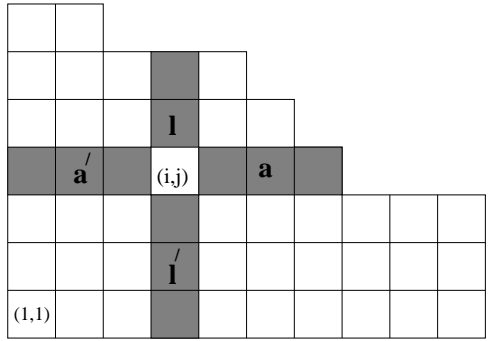}
\caption{The arm $a$, coarm $a^{\prime}$, leg $l$ and coleg $l^{\prime}$ of a cell.}
\label{legarm}
\end{center}
\end{figure}

Given any symmetric function $f$, let $\Delta ^{\prime}_f$ be the linear operator defined on the modified Macdonald basis ${\widetilde H}_{\mu}(X;q,t)$ as
\begin{align}
\Delta^{\prime} _{f} {\widetilde H}_{\mu}(X;q,t) =  f[B_{\mu}(q,t)-1] {\widetilde H}_{\mu}(X;q,t),
\end{align}
where by $f[B_{\mu}(q,t)-1]$ we mean the result of substituting the elements $q^{a'(c)}t^{l'(c)}$, $c\in \mu$, $c \ne (1,1)$, for the variables in the definition of $f$.   This is the simplest example of a 
{\it plethystic substitution}, which we define in Section \ref{Prelim}. Throughout this article, when dealing with a plethystic substitution of some alphabet $E$ into a symmetric function $f$, we will place $E$ inside square brackets, as in $f[E]$, as opposed to placing $E$ inside parentheses.  We will also use the standard notation $e_n(X)$ for the $n$th elementary symmetric function in the set of variables $X=\{x_1,x_2,\ldots \}$.

The Delta Conjecture from \cite{HRW15} says that for any $1\le k \le n$, 
\begin{align}
\label{DeltaP}
\Delta ^{\prime} _{e_{k-1}} e_n (X) = 
\sum_{\pi } t^{ \text{area} (\pi) } \sum_{ \sigma \in \text{WP}(\pi) } x^{ \sigma } q^{ \text{dinv} (\sigma) }\prod_{a_i>a_{i-1}} (1+z/t^{a_i})\Big|_{z^{n-k}},
\end{align}
where $\pi$ is a Dyck path (a lattice path from $(0,0)$ to $(n,n)$ consisting of unit North and East steps, which never goes below the line $y=x$), $\sigma$ is a 
``word parking function" for $\pi$ (a labelling on the North steps of $\pi$ with positive integers which is increasing up columns), and $\text{dinv}$, $\text{area}$ are
statistics on these objects with simple combinatorial descriptions.  The product on the right-hand-side of (\ref{DeltaP}) is over all pairs of consecutive North steps of $\pi$,
$a_i$ is the number of area cells in the $i$th row of $\pi$, and $|_{z^j}$ means ``take the coefficient of $z^j$ in".  
We refer the reader to \cite{HRW15} for precise descriptions of these concepts; all we will need here is what happens to the
two sides of (\ref{DeltaP}) when we set $q=0$, which we describe in Section \ref{sec:algebraic}.  

When $k=n$ the Delta Conjecture reduces to the well-known Shuffle Conjecture from \cite{HHLRU05} (now a theorem of Carlsson and Mellit \cite{CM}).  
The techniques used to prove the Shuffle Conjecture though do not seem to apply immediately
when $k<n$.  Romero \cite{Rom17} proved the Delta Conjecture when $q=1$.  Zabrocki \cite{Zab16} proved one of two conjectures from 
\cite{HRW15} involving the coefficient of a hook Schur function in (\ref{DeltaP}), and in recent work D'Adderio and  Wyngaerd \cite{DaWy17} prove the other of these conjectures.
In \cite{HRS} it is proved that the right-hand-side of (\ref{DeltaP}), when $q=0$, is the graded
Frobenius characteristic of a certain symmetric-group module.  

We will often refer to the left-hand-side of (\ref{DeltaP}) as $\text{SF}(X;q,t)$ (the ``symmetric function side"), and the right-hand-side as $\text{Rise}_{n,k}(X;q,t)$ (the ``combinatorial side").  
Actually, the Delta Conjecture says that $\text{SF}(X;q,t)=\text{Rise}_{n,k}(X;q,t)=\text{Val}_{n,k}(X;q,t)$, where $\text{Val}_{n,k}(X;q,t)$ has a combinatorial description similar to $\text{Rise}_{n,k}(X;q,t)$.  
It is still unknown if any of the three functions $\text{Val}(X;q,t)$, $\text{Rise}_{n,k}(X;q,t)$, or $\text{SF}(X;q,t)$ equal each other, but results in  \cite{ReWi15}, \cite{Wilson}, \cite{HRW15}, and 
\cite{Rhoades} imply that $\text{Val}_{n,k}(X;q,0)=\text{Val}_{n,k}(X;0,q)=\text{Rise}_{n,k}(X;q,0)=\text{Rise}_{n,k}(X;0,q)$.
From its expansion in terms of Macdonald polynomials derived using \eqref{eqn:en}, it is clear that $SF(X;q,t)=SF(X;t,q)$.  
It follows that the $q=0$ case of the Delta Conjecture implies the case $t=0$.  
In the following sections we prove that $\text{SF}(X;0,q)=\text{Rise}_{n,k}(X;0,q)$,  which thus proves both the case $q=0$ and the case $t=0$ of the Delta Conjecture, which were open until now.

Finally it is worth mentioning that our proof 
of the Delta  conjecture at $q=0$ uses a novel method of proving symmetric function identities. The method consists in expanding both sides of the identity to be proved as a linear combination of special evaluations of a ``Cauchy  Kernel'' naturally  associated with  the problem. This method has recently been succesfully used in other closely related problems involving Delta operators.

%------------------------------------------------------------------------------------------------------

\section{Preliminaries}
\label{Prelim}

%------------------------------------------------------------------------------------------------------

In this section, we define necessary notation and introduce backgrounds. We basically follow the notation and terminology for symmetric functions in \cite{Mac}.

A \emph{partition} of $n$, denoted by $\lambda\vdash n$, is a sequence of weakly decreasing  positive integers 
$$\lambda = (\lambda_1,\lambda_2, \dots, \lambda_k)$$
such that $\sum_{i=1}^k \lambda_i =n$. Each $\lambda_i$ is called a \emph{part} and the number of parts is the \emph{length} of $\lambda$, 
denoted by $\ell(\lambda)$. For each $i$, $m_i (\lambda)$ counts the number of times $i$ occurs as a part of $\lambda$, called \emph{multiplicity} of $i$ in $\lambda$. 
Given a partition $\lambda$, we define 
$$n(\lambda) = \sum_{i\ge 1}(i-1)\lambda_i =\sum_{i\ge 1}\binom{\lambda_i '}{2}=\sum_{c\in \lambda}l(c)=\sum_{c\in \lambda}l'(c).$$
The \emph{conjugate} of a partition $\lambda$ is a partition $\lambda'$ whose parts are 
$$\lambda_i' = |\{ j ~:~ \lambda_j \ge i\}|,$$
or whose multiplicities $m_i (\lambda')=\lambda_i -\lambda_{i+1}$. 

For a partition $\lambda$ of $n$, we let 
$$m_\lambda (X),\quad e_\lambda (X),\quad h_\lambda (X),\quad p_\lambda(X),\quad s_\lambda (X),\quad\widetilde{H}_\lambda (X;q,t)$$
denote the \emph{monomial}, \emph{elementary}, \emph{complete homogeneous}, \emph{power sum}, \emph{Schur} and \emph{modified Macdonald symmetric functions}, respectively. 

In dealing with symmetric function identities, especially with those arising in the theory of Macdonald polynomials, we find it convenient and often indispensable to 
use plethystic notation and so here we define \emph{plethystic substitution}. 
Let $E=E(t_1,t_2,\dots)$ be a formal Laurent series with rational coefficients in indeterminates $t_1$, $t_2$, $\dots$. 
For the $k$th power sum symmetric function $p_k (X)= \sum_i x_i^k$, we define 
$$p_k [E] = E(t_1 ^k, t_2 ^k, \dots ).$$
This given, for any symmetric function $f$, we set 
$$f[E]\ses Q_f(p_1,p_2, \ldots )\Big|_{p_k\RA E(t_1^k,t_2^k,\ldots )}
$$
where $Q_f$ is the polynomial yielding the expansion of $f$ in 
terms of the power sum basis.
The convention is that in a plethystic expression $X$ stands for the sum of the original indeterminates $x_1 +x_2 +\cdots $, since $p_k [X] = p_k (X)$. 
Also, note that $p_k [X-Y] = p_k [X] - p_k[Y]$. Hence, for example, by
$p_k[X(1-t)]$ we mean $p_k(X)(1-t^k)$.  
In particular we have
\begin{align}
\label{eps}
p_k\left [ \epsilon X \right ] &=(-1)^k p_k(X) \\
\notag
f \left[ -\epsilon X \right ] &=\omega f(X),
\end{align}
for any symmetric function $f(X)$, where $\epsilon$ denotes a special symbol for replacing variables by their negatives inside the plethystic brackets and $\omega$ is the usual involution on symmetric functions which acts on the Schur basis as
$\omega s_{\lambda}(X) = s_{\lambda ^{\prime}}(X)$.
We refer readers unfamiliar with these concepts to Chapters $1$ and $2$ of \cite{Hag08} for background and more details on plethysm and Macdonald polynomials.  

To make our notations simple, we let 
\begin{align}\label{eqn:notations}
\begin{split}
M &= (1-q)(1-t),\\
B_\mu (q,t)& =\sum_{c\in \mu}q^{a'(c)}t^{l'(c)},\\
 \Pi_\mu ^{\prime}(q,t)& = \prod_{\substack{c\in \mu\\ c\ne (0,0)}}(1-q^{a^{\prime}(c)}t^{l^{\prime}(c)}),\\
w_\mu (q,t) &= \prod_{c\in \mu}(q^{a(c)} -t^{l(c)+1})(t^{l(c)} -q^{a(c)+1}).
\end{split}
\end{align}
We also define 
$$\Omega (X)=\prod_i \frac{1}{1-x_i}=\exp \left(\sum_{k\ge 1}\frac{p_k (X)}{k} \right).$$
The following two identities will be useful in the proof of the Delta Conjecture when $q=0$:
\begin{align}
& \text{(Cauchy formula) }\quad \Omega\left[ -\epsilon \frac{XY}{M}\right] =\sum_{\mu}\frac{\widetilde{H}[X;q,t]\widetilde{H}[Y;q,t]}{w_\mu (q,t)},\label{eqn:Cauchy}\\
& \text{(Koornwinder-Macdonald reciprocity formula) }\label{eqn:reciprocity}\\
& \qquad\qquad\qquad \frac{\widetilde{H}_\mu [1+u(MB_\lambda (q,t) -1);q,t]}{\prod_{c\in \mu}(1-u q^{a'(c)}t^{l'(c)})} =  \frac{\widetilde{H}_\lambda [1+u(MB_\mu (q,t) -1);q,t]}{\prod_{c\in \lambda}(1-u q^{a'(c)}t^{l'(c)})}.\qquad\qquad\qquad\nonumber
\end{align}
The following lemmas will be used in the next section.
\begin{lem}\cite[Theorem 3.4]{GH96}\label{lem:en}
$$e_n (X) = \sum_{\mu\vdash n}\frac{(1-q)(1-t)\widetilde{H}_\mu (X;q,t) \Pi_\mu ^{\prime}(q,t)B_\mu (q,t) }{w_\mu (q,t)}.$$
\end{lem}

\begin{proof}
In the Cauchy formula \eqref{eqn:Cauchy}, if we read out the homogeneous component of degree $n$ in $X$ and $Y$ we get 
\begin{equation}\label{eqn:en_pf}
e_n \left[\frac{XY}{M} \right] =\sum_{\mu\vdash n}\frac{\widetilde{H}_\mu [X;q,t]\widetilde{H}_\mu [Y;q,t]}{w_\mu (q,t)}.
\end{equation}
Thus we only need to compute $\widetilde{H}_\mu [M;q,t]$ and let $Y=M$ in \eqref{eqn:en_pf}. 
To this end, we consider the Koornwinder-Macdonald reciprocity formula \eqref{eqn:reciprocity}.
If we cancel the common factor $(1-u)$ out of the denominator on both sides of \eqref{eqn:reciprocity} and set $u=1$
we obtain 
$$\frac{\widetilde{H}_\mu [MB_\lambda (q,t);q,t]}{\Pi_{\mu}'(q,t)}=\frac{\widetilde{H}_\lambda [MB_\mu (q,t);q,t]}{\Pi_{\lambda}'(q,t)}.$$
In the case when $\lambda=(1)$, the above identity reduces to 
$$\widetilde{H}_\mu [M;q,t] = MB_\mu (q,t) \Pi_\mu '(q,t)$$
which finishes the proof.
\end{proof}

\begin{lem}\label{lem:H}
For any monomial $u$, we have 
\begin{equation}\label{eqn:lem-2}
H_\mu [1-u;q] =  q^{n(\mu)}\prod_{j=0}^{\l(\mu)-1}(1-u/q^{j})=q^{n(\mu)-\binom{\l(\mu)}{2}}\prod_{j=0}^{\l(\mu)-1}(q^{j}-u).
\end{equation}
\end{lem}

\begin{proof}
For $\lambda$ the empty partition, $\widetilde{H}_\lambda(X;q,t)=1$ and $MB_\lambda (q,t)-1=-1$, so 
\eqref{eqn:reciprocity} in this case reduces to 
 $\widetilde{H}_\mu [1-u;q,t]=\prod_{c\in\mu}(1-ut^{l'(c)}q^{a'(c)})$.
 Since $\widetilde{H}_\mu (X;q,t)=t^{n(\mu)}H_\mu (X;q,t^{-1})$, 
\begin{align*}
H_\mu [1-u;q,t] &= t^{n(\mu)}\prod_{c\in \mu}(1-u t^{-l'(c)}q^{a'(c)})\\
&= t^{n(\mu)}\prod_{\substack{c\in \mu,\\ a'(c)=0}}(1-u t^{-l'(c)})
\prod_{\substack{c\in \mu,\\ a'(c)>0}}(1-u t^{-l'(c)}q^{a'(c)}).
\end{align*}
Setting $q=0$ proves the lemma after replacing $t$ by $q$, since 
$$
H_\mu [1-u;t] = H_\mu [1-u;0,t] = t^{n(\mu)}\prod_{j=1}^{\l(\mu)}(1-u t^{-(j-1)})
= t^{n(\mu)-\binom{\l(\mu)}{2}}\prod_{j=1}^{\l (\mu)}(t^{j-1}-u).
$$
\end{proof}

%-------------------------------------------------------------------------------------------------------

\section{Schur expansion of $\Delta_{e_{k-1}} ' e_n$ when $q=0$}\label{sec:algebraic}

%-------------------------------------------------------------------------------------------------------

Let 
$$
C_{n,k}(X;t) = \text{Rise}_{n, k}(X;0,t),$$
we have the following:
\begin{lem}\cite[Theorem 6.11, 6.14]{HRS}\label{lem:rhs}
$$\omega(C_{n, k}(X;t))=\sum_{\substack{\mu\vdash n \\ \ell (\mu) = k}}t^{n(\mu)-\binom{k}{2}}\begin{bmatrix}k\\ m_1 (\mu), m_2 (\mu), \dots, m_n (\mu) \end{bmatrix}_t Q_\mu (X;t),$$
where
$$
Q_{\mu}(X;t) = \sum_{\lambda} s_{\lambda}(X)K_{\lambda,\mu}(t).
$$ 
Equivalently, for any $\lambda$ we have
\begin{equation}\label{eqn:Cnk}
\langle C_{n, k}(X;t), s_\lambda\rangle = \sum_{\substack{\mu\vdash n \\ \ell(\mu) = k}}t^{n(\mu)-\binom{k}{2}}K_{\lambda ^{\prime}, \mu}(t)\begin{bmatrix}k\\ 
m_1 (\mu),  m_2 (\mu), \dots,  m_n (\mu)\end{bmatrix}_t.
\end{equation}
\end{lem}

\begin{proof}
This is a direct result of combining Theorem 6.11 and Theorem 6.14 of \cite{HRS}.  See \cite{HRS} for details.
\end{proof}

\begin{prop}\label{prop:deltaconj_0}
The Delta Conjecture when $q=0$ is equivalent to the following identity 
\begin{multline}\label{eqn:deltaconj_0}
\sum_{\mu\vdash n} (-1)^{n-\ell(\mu)}q^{-n-n(\mu)+\sum_{i=1}^n \binom{m_i +1}{2}}\begin{bmatrix} \ell(\mu) -1\\k-1\end{bmatrix}_q 
\frac{(q;q)_{\ell(\mu)}}{\prod_{i=1}^n (q;q)_{m_i (\mu)}} K_{\lambda,\mu}(q^{-1})\\
=q^{-k(k-1)}\sum_{\mu\vdash n \atop \ell(\mu)=k }q^{n(\mu)} \frac{(q;q)_{\ell(\mu)}}{\prod_{i=1}^n (q;q)_{m_i (\mu)}}
K_{\lambda ',\mu} (q)
\end{multline}
for all $\lambda \vdash n$ and $1\le k\le n$. 
\end{prop}

\begin{proof}
Observe that by \eqref{eqn:Cnk}, the right hand side of \eqref{eqn:deltaconj_0} equals $q^{-{k \choose 2}} \langle C_{n, k}(X;q), s_\lambda\rangle$.
 We obtain 
the left hand side of \eqref{eqn:deltaconj_0} by computing $\left.\langle\Delta_{e_{k-1}} ' e_n, s_\lambda\rangle\right|_{\substack{q=0\\ t\rightarrow q}}$ algebraically.\\
Recall that by Lemma \ref{lem:en},
\begin{equation}\label{eqn:en}
e_n (X) = \sum_{\mu\vdash n}\frac{(1-q)(1-t)\widetilde{H}_\mu (X;q,t) \Pi_\mu ^{\prime}(q,t)B_\mu (q,t) }{w_\mu (q,t)}.
\end{equation}
Applying $\Delta_{e_{k-1}}^{\prime}$ to both sides of (\ref{eqn:en}) we get 
\begin{equation}\label{eqn:Delta_en}
\Delta_{e_{k-1}} ^{\prime} e_n = \sum_{\mu\vdash n}\frac{(1-q)(1-t)\widetilde{H}_\mu (X;q,t) B_\mu(q,t) \Pi_\mu ^{\prime}(q,t) e_{k-1}[B_\mu -1]}{w_\mu}.
\end{equation}
If we let $q=0$ in \eqref{eqn:notations} we have 
\begin{align*}
B_\mu (0,t) &= 1+t+\cdots +t^{\ell (\mu)-1},\\
 \Pi_\mu ' (0,t) &= (1-t)\cdots (1-t^{\ell (\mu)-1}),\\
w_\mu (0,t) &= \prod_{c\in \mu}t^{l(c)} \cdot \prod_{\substack{c\in \mu\\ a(c)=0}}(1-t^{l(c)+1})\cdot \prod_{\substack{c\in \mu \\ a(c)>0}}(-t^{l(c)+1})\\
&= (-1)^{n-\ell (\mu)}t^{2n(\mu)+n-\sum_i \binom{m_i +1}{2}}\prod_{i}\varphi_{m_i}(t),\\
e_{k-1} [B_\mu (0,t)-1]&= e_{k-1} (t, t^2, \dots, t^{\ell (\mu)-1})=t^{k-1+\binom{k-1}{2}}\begin{bmatrix} \ell (\mu)-1\\ k-1\end{bmatrix}_t,
\end{align*}
where $\varphi_m (t)=(t;t)_m = (1-t)\cdots (1-t^m)$. 
We use the above computations in \eqref{eqn:Delta_en} to obtain 
\begin{multline*}
\left.\Delta_{e_{k-1}} ^{\prime}e_n\right|_{\substack{q=0\\ t\rightarrow q}} = \sum_{\mu\vdash n}\frac{(1-q) \widetilde{H}_\mu (X;0,q) (1+q+\cdots +q^{\ell(\mu)-1})(q;q)_{\ell(\mu)-1}q^{k-1+\binom{k-1}{2}}\begin{bmatrix} \ell(\mu)-1 \\ k-1\end{bmatrix}_q}{(-1)^{n-\ell(\mu)}q^{2n(\mu)+n-\sum_i \binom{m_i +1}{2}}\prod_{i}\varphi_{m_i}(q)}\\
= \sum_{\mu\vdash n}(-1)^{n-\ell(\mu)}q^{k-1+\binom{k-1}{2}-2n(\mu)-n+\sum_i \binom{m_i +1}{2}}\widetilde{H}_\mu (X;0,q) \begin{bmatrix} \ell (\mu)-1\\ k-1\end{bmatrix}_q\\ \times \begin{bmatrix} \ell (\mu)\\ m_1 (\mu), m_2 (\mu), \ldots , m_n (\mu)\end{bmatrix}_q.
\end{multline*}
Hence,
\begin{multline*}
\left.\langle \Delta_{e_{k-1} } ^{\prime} e_n, s_\lambda\rangle \right|_{\substack{q=0\\ t\rightarrow q}}\\
=\sum_{\mu\vdash n}\widetilde{K}_{\lambda , \mu}(q) (-1)^{n-\ell (\mu)}q^{\binom{k}{2}-2n(\mu)-n+\sum_i \binom{m_i +1}{2}} \begin{bmatrix} \ell (\mu)-1\\ k-1\end{bmatrix}_q \begin{bmatrix} \ell (\mu)\\ m_1 (\mu), m_2 (\mu), \dots, m_n (\mu)\end{bmatrix}_q.
\end{multline*}
Since $\widetilde{K}_{\lambda , \mu}(q) = q^{n(\mu)}K_{\lambda,\mu}(q^{-1})$, we eventually have 
\begin{multline*}
\left.\langle \Delta_{e_{k-1} } ^{\prime} e_n, s_\lambda\rangle \right|_{\substack{q=0\\ t\rightarrow q}}\\
=\sum_{\mu\vdash n}(-1)^{n-\ell (\mu)}q^{\binom{k}{2}-n(\mu)-n+\sum_i \binom{m_i +1}{2}} \begin{bmatrix} \ell (\mu)-1\\ k-1\end{bmatrix}_q 
\frac{(q;q)_{\ell(\mu)}}{\prod_{i=1}^n (q;q)_{m_i (\mu)}} K_{\lambda,\mu}(q^{-1}).
\end{multline*}
Then \eqref{eqn:deltaconj_0} is the result of multiplying both sides of 
$$ \left.\langle\Delta_{e_{k-1}} ' e_n, s_\lambda\rangle\right|_{\substack{q=0\\ t\rightarrow q}}=\langle C_{n, k}(X;q), s_\lambda\rangle,$$
by  $q^{-\binom{k}{2}}$.
\end{proof}

%------------------------------------------------------------------------------------------------------

\section{The proof}

%------------------------------------------------------------------------------------------------------

In this section, we prove the Delta Conjecture when $q=0$ by verifying the equivalent identity in \eqref{eqn:deltaconj_0}.

\begin{thm}\label{thm:deltaconj_0}
The Delta Conjecture is true when $q=0$. 
\end{thm}

\begin{proof}
In Proposition \ref{prop:deltaconj_0}, we showed that the Delta Conjecture at $q=0$ is equivalent  to the following identities 
\begin{equation}
LHS_{k,\la}\ses RHS_{k,\la}
\bigsp 
(\hbox{for all $\la\part n$ and $1\le k\le n$})
\label{eqn:pf-1}
\end{equation}
where
\begin{equation}
LHS_{k,\la}= 
\sum_{\mu\part n}
(-1)^{n-\ell(\mu)}
\Big[{ \ell(\mu)-1\atop k-1 }\Big]_q
q^{\sum_{i=1}^{n}({m_i+1\atop 2})- n(\mu)}
{q^{-n} (q;q)_{\ell(\mu)}\over \prod_{i=1}^n (q;q)_{m_i(\mu)}
}\ssp K_{\la,\mu}(1/q)\label{eqn:pf-2}
\end{equation}
and
\begin{equation}
RHS_{k,\la}= q^{-k(k-1) }
\sum_{\substack{\mu\part n, \\ \ell(\mu)=k }}
q^{n(\mu)}
{(q;q)_{\ell(\mu)}\over \prod_{i=1}^n (q;q)_{m_i(\mu)}
}
\ssp K_{\la',\mu}(q).\label{eqn:pf-3}
\end{equation}
We first eliminate the presence of the expressions
$$\prod_{i=1}^n (q;q)_{m_i(\mu)}$$
which make the equality in \eqref{eqn:pf-1} combinatorially forbidding. 
We note the following identity in \cite[III, (2.11)]{Mac}
\begin{equation}
P_\mu[X;q]\ses
 {1\over  \prod_{i=1}^{n}
(q;q)_{m_i(\mu)}
}\ssp Q_\mu[X;q]\label{eqn:pf-5}.
\end{equation}
The idea is to translate the equality in \eqref{eqn:pf-1} into a symmetric function identity involving the basis $
\big\{ P_\mu[X;q]\big\}_\mu 
$. Another important ingredient we find in Macdonald's book is the \emph{Hall-Littlewood Cauchy identity} \cite[III. (4.4)]{Mac}
\begin{equation}
\sum_{\mu\part n} P_\mu[X;q]Q_\mu[Y;q]\ses
\sum_{\rho\part n} {1\over z_\rho}p_\rho[XY(1-q)]\ses h_n[XY(1-q)]\label{eqn:pf-6}.
\end{equation}
The presence of the Kostka-Foulkes polynomials in both \eqref{eqn:pf-1} and \eqref{eqn:pf-2} suggests using the identity
\begin{equation}
Q_\mu[X;q]\ses \sum_{\la\ge \mu}s_\la[X(1-q)]K_{\la,\mu}(q)\label{eqn:pf-7}.
\end{equation}
Now if we define 
\begin{equation}
H_\mu[X;q]\ses Q_\mu[\tttt{X\over 1-q};q]
\ses \sum_{\la\ge \mu}s_\la[X]K_{\la,\mu}(q),\label{eqn:pf-8}
\end{equation}
then \eqref{eqn:pf-6} can be rewritten as the ``Cauchy Kernel'' 
\begin{equation}
\sum_{\mu\part n} P_\mu[X;q]H_\mu[Y;q]
\, =\,  h_n[XY]\label{eqn:pf-9}.
\end{equation}
This is particularly enticing since 
it remains true under the replacement of $q$ by $1/q$. 
Using \eqref{eqn:pf-7} and \eqref{eqn:pf-8}, we can rewrite \eqref{eqn:pf-9} as 
$$
\sum_{\mu\part n \atop \lambda \ge \mu} P_\mu[X;q] \sum_{\la\part n}
s_\la[Y]K_{\la,\mu}(q) 
\ses  h_n[XY]
$$
or better
$$
 \sum_{\la\part n}s_\la[Y]
 \sum_{\mu\le\la}
 K_{\la,\mu}(q)P_\mu[X;q]
 \ses h_n[XY],
$$
forcing the identity
\begin{equation}
s_\la[X]\ses \sum_{\mu\le \la}K_{\la,\mu}(q)P_\mu[X;q].\label{eqn:pf-10}
\end{equation}
The plan is to express $\sum_{\lambda\vdash n} LHS_{k,\lambda}s_{\lambda'}[X(1-q)]$ and 
$\sum_{\lambda\vdash n} RHS_{k,\lambda}s_{\lambda'}[X(1-q)]$ as identities in terms of the $P_{\mu}$'s. 
We first work with the right hand side of \eqref{eqn:pf-2}.  Using the fact that
$$
(q;q)_m\Big|_{q\RA {1\over q}}=\,\,
(1-1/q)(1-1/q^2)\cdots (1-1/q^m)\ses 
(-1)^m {(q;q)_m\over  q^{m+1 \choose 2}},
$$
\eqref{eqn:pf-2} can be rewritten as 
$$
LHS_{k,\la}=\, 
\sum_{ \mu\part n }
q^{- n(\mu)}
{1\over \prod_{i=1}^n (1/q;1/q)_{m_i(\mu)}
}\ssp  K_{\la,\mu}(1/q)(-q)^{-n} \Big[{ \ell(\mu)-1\atop k-1 }\Big]_q (q;q)_{\ell(\mu)}.
$$
Multiplying by $s_{\la'}[x(1-q)]$ and summing for all  $\la\part n$ gives
$$
\sum_{\la\part n}LHS_{k,\la}s_{\la'}[x(1-q)]=\,  
\sum_{ \mu\part n }
q^{- n(\mu)}
{\sum_{\la\part n}s_{\la'}[x(1-q)]K_{\la,\mu}(1/q)\over \prod_{i=1}^n (1/q;1/q)_{m_i(\mu)}
}(-q)^{-n} \Big[{ \ell(\mu)-1\atop k-1 }\Big]_q (q;q)_{\ell(\mu)}.
$$
But since
$s_{\la'}[x(1-q)]=q^n s_{\la'}[x(1/q-1)]=
  (-q)^n s_{\la}[x(1-1/q)]$, we obtain
$$
\sum_{\la\part n}LHS_{k,\la}s_{\la'}[x(1-q)]=\, 
\sum_{ \mu\part n }
q^{- n(\mu)}
{\sum_{\la\part n}s_{\la }[x(1-1/q)]K_{\la,\mu}(1/q)\over \prod_{i=1}^n (1/q;1/q)_{m_i(\mu)}
}  \Big[{ \ell(\mu)-1\atop k-1 }\Big]_q (q;q)_{\ell(\mu)}, 
$$
or better, using \eqref{eqn:pf-7},
$$
\sum_{\la\part n}LHS_{k,\la}s_{\la'}[x(1-q)]=\, 
\sum_{ \mu\part n }
q^{- n(\mu)}
{
Q_\mu[X;1/q]
\over \prod_{i=1}^n (1/q;1/q)_{m_i(\mu)}
}  \Big[{ \ell(\mu)-1\atop k-1 }\Big]_q (q;q)_{\ell(\mu)}.
$$
Thus \eqref{eqn:pf-5}  gives 
\begin{equation}
\sum_{\la\part n}LHS_{k,\la}s_{\la'}[x(1-q)]=
\sum_{\mu\part n}q^{- n(\mu)}
P_\mu[X;1/q]   
\Big[{ \ell(\mu)-1\atop k-1 }\Big]_q
(q;q)_{\ell(\mu)}.\label{eqn:pf-12}
\end{equation}

Let us now work on the other identity \eqref{eqn:pf-3}. 
Replicating what we did for \eqref{eqn:pf-2}, we multiply by 
 $s_{\la'}[x(1-q)]$ and sum over $\la\part n$  to get
\begin{align*}
\sum_{\la\part n}RHS_{k,\la}s_{\la'}[X(1-q)]\, &=\,
q^{-k(k-1) }
(q;q)_k 
\sum_{\substack{\mu\part n, \\ \ell(\mu)=k}} 
q^{n(\mu)}
{
\sum_{\la\part n}s_{\la'} [X(1-q)]K_{\la',\mu}(q)
\over \prod_{i=1}^n (q;q)_{m_i(\mu)}}\\
&= q^{-k(k-1)} 
(q;q)_k \sum_{\substack{\mu\part n, \\ \ell(\mu)=k}}
q^{n(\mu)}{ Q_\mu[X;q]\over \prod_{i=1}^n (q;q)_{m_i(\mu)}}\\
&= q^{-k(k-1) }
(q;q)_k\sum_ {\substack{\mu\part n,\\ \ell(\mu)=k }}
q^{n(\mu)}
  P_\mu[X;q],
\end{align*}
using \eqref{eqn:pf-7} and \eqref{eqn:pf-5} for the second and the third identities, respectively.
Thus, we are left to match up the following two identities 
\begin{equation}
\sum_{\la\part n}LHS_{k,\la}s_{\la'}[x(1-q)]=
\sum_{\mu\part n}q^{- n(\mu)}
P_\mu[X;1/q]   
\Big[{ \ell(\mu)-1\atop k-1 }\Big]_q
(q;q)_{\ell(\mu)}\label{eqn:pf-13}
\end{equation}
and 
\begin{equation}
\sum_{\la\part n}RHS_{k,\la}s_{\la'}[X(1-q)]\ses
q^{-k(k-1)}
(q;q)_k\sum_{\substack{\mu\part n\\ \ell(\mu)=k}}
q^{n(\mu)}
 P_\mu[X;q].\label{eqn:pf-14}
\end{equation}
The idea is to use the following two identities to evaluate \eqref{eqn:pf-13} and \eqref{eqn:pf-14} respectively :
\begin{equation}
\text{(a)}\ess\ess \sum_{\mu\part n} P_\mu[X;q]H_\mu[Y;q]
= h_n[XY],
\bigsp
\text{(b)}\ess\ess \sum_{\mu\part n} P_\mu[X;1/q]H_\mu[Y;1/q]
= h_n[XY].\label{eqn:pf-15}
\end{equation}
Computer experiments computing special cases of \eqref{eqn:pf-14} showed us that 
the Schur expansion of \eqref{eqn:pf-14} only contained the Schur functions indexed by hook shapes. 
The symmetric function identity which could explain this phenomenon is the following 
(valid for any monomial $u$) :
$$
s_\la[1-u]\ses \begin{cases}
(-u)^r(1-u) & \text{if $\la=n-r,1^r$ for $0\le r\le n-1$}\\
0 &  \text{otherwise.}
\end{cases}
$$ 
Hence we try the same substitution $Y=1-u$ in \eqref{eqn:pf-15}.

\begin{lem}\label{lem:h}
For all monomials $u$, we have 
\begin{equation}
h_n[X(1-u)] = \sum_{\mu\vdash n}P_{\mu}[X;q] H_{\mu} [1-u;q] =(1-u)\sum_{s=0}^{n-1} (-u)^s s_{(n-s,1^s)}[X].\label{eqn:lem-1}
\end{equation}
\end{lem}
\begin{proof}
Notice that for $Y=1-u$ \eqref{eqn:pf-15} (a) gives 
\begin{align*}
\sum_{\mu\vdash n}P_{\mu}[X;q] H_{\mu}[1-u;q] &= h_n [X-uX]\\
&= h_n [X] +\sum_{s=1}^{n-1} h_{n-s}[X] (-u)^s e_s [X] +(-u)^n e_n [X]\\
&= s_n [X] +\sum_{s=1}^{n-1} (-u)^s (s_{(n-s, 1^s)}[X] +s_{(n-s+1, 1^{s-1})}[X])+(-u)^n s_{(1^n)}[X]\\
&= (1-u)s_n [X] +\sum_{s=1}^{n-1} (-u)^s s_{(n-s, 1^s)}[X] +\sum_{s=1}^{n-1} (-u)^{s+1} s_{(n-s, 1^s)}[X]\\
&= (1-u)\sum_{s=0}^{n-1}(-u)^s s_{(n-s, 1^s)}[X].
\end{align*}
\end{proof}

Applying Lemma \ref{lem:H} in \eqref{eqn:pf-15} (a) with $u=q^i$ gives 
\begin{equation}
\frac{h_n [X(1-q^i)]}{1-q^i} =\sum_{\mu\vdash n}q^{n(\mu)}P_\mu [X;q] \prod_{j=1}^{\l (\mu)-1}(1-q^{i-j}).\label{eqn:pf-20}
\end{equation}
Lemma \ref{lem:h} tells us that taking the coefficients of powers of $u$ in the ratio $\ssp h_n[X(1-u)]/(1-u)\ssp $ successively yields  all the hook Schur functions. This fact suggests the hook Schur function expansion
would be explained if we can find coefficients
$c_i ^{(k)}(q)$  yielding the identity
\begin{equation}
\sum_ {\substack{\mu\part n, \\ \l(\mu)=k}}
q^{n(\mu)}
 P_\mu[X;q] =\sum_{i=1}^k c_i ^{(k)}(q)\frac{h_n[X(1-q^i)]}{1-q^i}.\label{eqn:pf-21}
\end{equation}
A procedure was thus constructed that determined the $c_i ^{(k)}(q)'s$ from the equations obtained by equating the two sides of \eqref{eqn:pf-21}. 
There we encountered another incredible surprise: these coefficients depended only on $k$ and not on $n$. 
Moreover, in all instances explored the  $c_i ^{(k)}(q)$'s turned out to be products of 
$q$-analogues of integers. 
This evidence prompted us to prove this new phenomenon by a close examination of the right hand side of \eqref{eqn:pf-21}. 
To this end notice first that using \eqref{eqn:pf-20} we can rewrite \eqref{eqn:pf-21} as
\begin{align}
\sum_ {\mu\part n}
q^{n(\mu)}
 P_\mu[X;q] \chi\big(\l(\mu)=k\big)
 &= \sum_{i=1}^k c_i ^{(k)}(q) \sum_{\mu\vdash n}q^{n(\mu)}P_{\mu}[X;q]\prod_{j=1}^{\l (\mu)-1}(1-q^{i-j})\nonumber\\
 &=\sum_{\mu\part n} q^{n(\mu)} P_\mu[X;q]\ssp
\sum_{i=1}^k c_i ^{(k)}(q)\prod_{j=1}^{\l(\mu)-1}
(1- q^{i-j}  ),\label{eqn:pf-22}
\end{align}
by changing the order of summation. Notice that to make the factors in $\prod_{j=1}^{\l(\mu)-1}
(1- q^{i-j}  )$ nonzero, we should have $i\ge \l(\mu)$. We use the $q$-Pochhammer symbol to denote  $\prod_{j=1}^{\l(\mu)-1}
(1- q^{i-j}  )$ in a compact form, namely, 
$$ \prod_{j=1}^{\l(\mu)-1} (1- q^{i-j}  ) = (q^{i+1-\l (\mu)};q)_{\l (\mu)-1}.$$
Since the family of symmetric polynomials 
$\big \{P_\mu[X;q] \big\}_{\mu\part n}$ is a basis, \eqref{eqn:pf-22} is true if and only if
\begin{equation}
\sum_{i=h}^k c_i ^{(k)}(q) (q^{i+1-h};q)_{h-1}=\chi (h=k),\qquad\text{ for all } 1\le h \le k.\label{eqn:pf-23}
\end{equation}
More importantly, we can easily  see that this system of equations is a triangular in the unknown coefficients $c_i(q)$. 
Thus the $c_i(q)$ exist and are unique.

A close look at the computer data revealed the 
 following explicit formulas for $c_i ^{(k)}(q)$.

\begin{lem}\label{lem:c_i}
\begin{equation}
c_i^{(k)}(q)= (-1)^{k-i}q^{k-i\choose 2}
\Big[{k-1 \atop i-1}
\Big]_q{ 1\over (q;q)_{k-1}}
=
{(-1)^{k-i}q^{k-i\choose 2} \over (q,q)_{k-i} 
 (q,q)_{i-1}}
\ess\ess\ess\ess \hbox{for $1\le i\le k$}.\label{eqn:c_i}
\end{equation}
\end{lem}

\begin{proof}
Since the equations in 
\eqref{eqn:pf-23} uniquely determine the $c_i ^{(k)}(q)$, to prove \eqref{eqn:c_i} it suffices to show that the $c_i ^{(k)}(q)$, as given by \eqref{eqn:c_i}, are  solutions of the equations in \eqref{eqn:pf-23}. Thus it is sufficient to show that
$$
\sum_{i=h}^k{(-1)^{k-i}q^{k-i\choose 2} \over (q,q)_{k-i} 
 (q,q)_{i-1}}\big(q^{i+1-h};q\big)_{h-1}
\ses\begin{cases}1 & \text{if } h=k\\ 0 & \text{if } h<k\end{cases}.
$$
Notice that this is trivially true for $h=k$. For $1\le h< k$, making the substitution 
$a=k-i$ we are reduced to showing that
$$
\sum_{a=0}^{k-h}
{(-1)^{a}q^{a\choose 2} \over (q,q)_{a} 
 (q,q)_{k-1-a}}\big(q^{k-a+1-h};q\big)_{h-1}
\ses 0.
$$
However, doing the necessary cancellations, this is none other than 
$$
\sum_{a=0}^{k-h}
(-1)^{a}q^{a\choose 2}{1 \over 
(q,q)_{a}(q,q)_{k-h-a} 
} 
\ses 0.
$$
Notice that if we multiply by $(q;q)_{k-h}$ the left hand side becomes 
$$
\sum_{a=0}^{k-h}
(-1)^{a}q^{a\choose 2}{ (q;q)_{k-h}\over 
(q,q)_{a}(q,q)_{k-h-a} 
} =\sum_{a=0}^{k-h} (-1)^a q^{\binom{a}{2}}\begin{bmatrix}k-h \\ a\end{bmatrix}_q 
$$
which is zero by the $q$-binomial theorem 
$$
(z;q)_n=(1-z)(1-zq)\cdots (1-zq^{n-1})
\, =\, \sum_{a=0}^{n}
(-z)^a q^{a\choose 2}\Big[{n\atop a} \Big]_q$$
with $z=1$. This completes the proof.
\end{proof}

Given \eqref{eqn:pf-21}, to verify that \eqref{eqn:pf-13} and \eqref{eqn:pf-14} are equal, we only need to show that 
\begin{equation}
\sum_{i=1}^k c_i ^{(k)}(q){h_n[X(1-q^i)]\over  1-q^i} \, =\,
{q^{k(k-1)}\over (q;q)_k}
\sum_{\mu\part n}q^{- n(\mu)}
P_\mu[X;1/q]   
\Big[{ \l(\mu)-1\atop k-1 }\Big]_q
(q;q)_{\l(\mu)},\label{eqn:pf-24}
\end{equation}
with $c_i ^{(k)}(q)$'s as given in Lemma \ref{lem:c_i}. 
To deal with the right hand side of \eqref{eqn:pf-24}, we utilize the identity \eqref{eqn:pf-15} (b) with the 
replacement $Y\RA (1-q^i)$  and obtain 
\begin{equation}
h_n[X(1-q^i)]\ses\sum_{\mu\part n} P_\mu[X;1/q]H_\mu[1-q^i;1/q]. \label{eqn:pf-30}
\end{equation}
By Lemma \ref{lem:H}, after replacing $q$ by $1/q$ and $u=q^i$, we have 
$$H_\mu [1-q^i;1/q] =q^{-n(\mu)}\prod_{j=0}^{\l (\mu)-1}(1-q^{i+j}).$$
Using this in \eqref{eqn:pf-30} gives 
\begin{equation}\label{eqn:pf-31}
{h_n[X(1-q^i)]\over  1-q^i}\ses
\sum_{\mu\part n} 
q^{-n(\mu)} P_\mu[X;1/q] 
\prod_{j=1}^{\l(\mu)-1}
(1-  q^{i+ j }  ).
\end{equation}
We use \eqref{eqn:pf-31} in the left hand side of \eqref{eqn:pf-24} to obtain an equivalent identity
$$
\sum_{i=1}^k c_i ^{(k)}(q)
\sum_{\mu\part n} 
q^{-n(\mu)} P_\mu[X;1/q] 
\prod_{j=1}^{\l(\mu)-1}
(1-  q^{i+ j }  )
\, =\,
{q^{k(k-1)}\over (q;q)_k}
\sum_{\mu\part n}q^{- n(\mu)}
P_\mu[X;1/q]   
\Big[{ \l(\mu)-1\atop k-1 }\Big]_q
(q;q)_{\l(\mu)},
$$
or better,
\begin{equation}
\sum_{\mu\part n} 
q^{-n(\mu)} P_\mu[X;1/q] 
\sum_{i=1}^k c_i ^{(k)}(q)\prod_{j=1}^{\l(\mu)-1}
(1-  q^{i+ j }  )
\, =\,
{q^{k(k-1)}\over (q;q)_k}
\sum_{\mu\part n}q^{- n(\mu)}
P_\mu[X;1/q]   
\Big[{ \l(\mu)-1\atop k-1 }\Big]_q
(q;q)_{\l(\mu)}.
\end{equation}
Since the family of polynomials 
$\big\{P_\mu(X;1/q)\big\}_{\mu\part n}$ is   also a basis, this can be true if and only if 
$$
\sum_{i=1}^k c_i ^{(k)}(q)\prod_{j=1}^{h-1}
(1-  q^{i+ j}  )\ses
{q^{k(k-1)}\over (q;q)_k}
\Big[{ h-1\atop k-1 }\Big]_q
(q;q)_{h}
$$
for $1\le h\le n$. Using the explicit expression for $c_i ^{(k)}(q)$,  this is none other than the equality  
\begin{equation}
\sum_{i=1}^k {(-1)^{k-i}q^{k-i\choose 2}(q^{i+1};q)_{h-1} \over (q,q)_{k-i} 
(q,q)_{i-1}}
\ses
{q^{k(k-1)}\over (q;q)_k}
\Big[{ h-1\atop k-1 }\Big]_q
(q;q)_{h}
\ess\ess\ess\ess\label{eqn:pf-29}
\end{equation}
for all $1\le h\le n$. 
With changing $k-i-1\RA ~i$ and the substitution $a=i-1$, \eqref{eqn:pf-29} becomes 
$$\sum_{a=0}^{k-1}\frac{(-1)^a q^{\binom{a}{2}}(q;q)_k (q;q)_{k-a+h-1}}{(q;q)_{k-a-1}(q;q)_a (q;q)_{k-a}(q;q)_h}
=q^{k(k-1)}\begin{bmatrix}h-1\\ k-1\end{bmatrix}_q.$$

\begin{prop}
\begin{equation}\label{eqn:id_pf}
\sum_{a=0}^{k-1}\frac{(-1)^a q^{\binom{a}{2}}(q;q)_k (q;q)_{k-a+h-1}}{(q;q)_{k-a-1}(q;q)_a (q;q)_{k-a}(q;q)_h}
=q^{k(k-1)}\begin{bmatrix}h-1\\ k-1\end{bmatrix}_q.
\end{equation}
\end{prop}

\begin{proof}
The left hand side of \eqref{eqn:id_pf} is equal to 
$$
\frac{(q;q)_k}{(q;q)_h}\sum_{a=0}^{k-1}\frac{(-1)^a q^{\binom{a}{2}} (q;q)_{k-a+h-1}}{(q;q)_{k-a-1}(q;q)_a (q;q)_{k-a}}
= \frac{(q;q)_{k+h-1}}{(q;q)_{h}(q;q)_{k-1}}\sum_{a=0}^{k-1} q^{a(k-h)} \frac{(q^{1-k};q)_a (q^{-k};q)_a}{(q;q)_a (q^{1-k-h};q)_a}.
$$
Now we apply the ${}_2\phi_1$ summation identity
$$\frac{(b;q)_m}{(c;q)_m}=\sum_{n=0}^m \frac{(q^{-m},c/b;q)_n}{(q,c;q)_n}(bq^m)^n,$$
with $m\mapsto k-1$, $b\mapsto q^{1-h}$ and $c\mapsto q^{1-k-h}$ to obtain
$$
\frac{(q^{1-h};q)_{k-1}}{(q^{1-k-h};q)_{k-1}}=\sum_{a=0}^{k-1}q^{a(k-h)}\frac{(q^{1-k};q)_a(q^{-k};q)_a}{(q;q)_a (q^{1-k-h};q)_a}.
$$
Hence
\begin{align*}
\frac{(q;q)_{k+h-1}}{(q;q)_h (q;q)_{k-1}}\cdot \frac{(q^{1-h};q)_{k-1}}{(q^{1-k-h};q)_{k-1}}
&= q^{k(k-1)}\frac{(q;q)_{h-1}}{(q;q)_{h-k}(q;q)_{k-1}}\\
&= q^{k(k-1)}\begin{bmatrix} h-1\\ k-1\end{bmatrix}_q.
\end{align*}
\end{proof}
This completes the proof of \eqref{eqn:pf-29} and the proof of the theorem.
\end{proof}

%---- reference  -------------------------------------------------------------------

%----------------------------------------------------------------------------------

\end{document}